\documentclass[review]{elsarticle}

\usepackage{lineno,hyperref}
\modulolinenumbers[5]
\usepackage{amsmath}
\usepackage{amssymb}
\usepackage{mathrsfs}
\usepackage{amsthm}
\usepackage{bm}
\usepackage{graphicx}
\usepackage{booktabs}
\usepackage{float}
\usepackage{geometry}
\usepackage{bbm}

\newcommand{\R}{\mathbb R}
\newcommand{\E}{\mathbb E}
\newcommand{\uP}{\mathbb P}
\newcommand{\ud}{\mathrm d}

\newcommand{\Var}{\mathrm{Var}}
\newcommand{\Cov}{\mathrm{Cov}}

\newtheorem{Def}{Definition}[section]
\newtheorem{lem}[Def]{Lemma}
\newtheorem{tho}[Def]{Theorem}
\newtheorem{prop}[Def]{Proposition}
\newtheorem{rem}[Def]{Remark}

\newtheorem{cor}[Def]{Corollary}

\begin{document}

\begin{frontmatter}

\title{Influences of Numerical Discretizations on Hitting Probabilities for Linear Stochastic Parabolic System}

%
%
%


\author[mymainaddress]{Chuchu Chen}
\ead{chenchuchu@lsec.cc.ac.cn}
\author[mymainaddress]{Jialin Hong}
\ead{hjl@lsec.cc.ac.cn}

\author[mymainaddress]{Derui Sheng\corref{mycorrespondingauthor}}
\cortext[mycorrespondingauthor]{Corresponding author}
\ead{sdr@lsec.cc.ac.cn}

\address[mymainaddress]{1. LSEC, ICMSEC, 
Academy of Mathematics and Systems Science, Chinese Academy of Sciences, 
Beijing, 100190, China\\ 
2. School of Mathematical Science, University of Chinese Academy of Sciences, Beijing, 100049, China}

\begin{abstract}
This paper investigates the influences of standard numerical discretizations on hitting probabilities for linear stochastic parabolic system driven by space-time white noises.
 We establish  lower and upper bounds for hitting probabilities of the associated numerical solutions of both temporal and spatial semi-discretizations in terms of  Bessel-Riesz capacity and Hausdorff measure, respectively. Moreover, the critical dimensions of  both  temporal  and spatial semi-discretizations turn out to be  half of those of the exact solution. 
This reveals that for a large class of Borel sets $A$, the probability of the event that the paths of the numerical solution hit $A$ cannot converge to that of the exact solution.

\end{abstract}

\begin{keyword}
hitting probability, numerical discretization,  stochastic parabolic system, critical dimension
\MSC[2010] 60H35\sep  65C30\sep 60J45
\end{keyword}

\end{frontmatter}

\linenumbers

\section{Introduction}

Hitting probability is an active field of probability potential theory. Generally speaking, for an $\R^d$-valued random field  $X=\{X(x),x\in \R^m\}$, its
hitting probability concerns the lower and upper bounds of $\mathbb{P}\{X(I)\cap A\neq \emptyset\}$, where $I\subset\R^m$ is a fixed compact set with positive Lebesgue measure, and   $A\subset \R^d$ is a Borel set.  
When $X$ is the solution of a system of SPDEs, the hitting probability of $X$ has been  studied extensively (see e.g., \cite{CA14,DKN07,DKN09,DKN13,DP21,DS10,DS15,HLS16,MT02,NV09,SV18}), which is usually bounded by the Bessel-Riesz capacity and Hausdorff measure of $A$, namely,
\begin{equation}\label{CH}
c\operatorname{Cap}_{d-Q}(A)  \le \uP\left\{X(I)\cap A \neq \emptyset\right\}  \le C \mathscr{H}_{d-Q}(A),
\end{equation}
for some $Q>0$ and $c,C>0$.  It is well known that the \textit{critical dimension} $Q$ is an important parameter that is highly related to the polarity of a Borel set $A$.  
When $X$ is approximated by a perturbation, particularly by a numerical solution  provided that $X$ is the exact solution of a system of SPDEs, a natural question is what are influences of the perturbation on critical dimensions of hitting probabilities.
 To the best of our knowledge, there is no result on this problem.

In this paper, we investigate influences on hitting probabilities of numerical discretizations for the following  linear stochastic parabolic system 
\begin{equation}\label{heat}
\left\{
\begin{split}
&\partial_tu^j(t,x)-\partial_{xx}u^j(t,x)=\dot{W}^j(t,x),\\
&u^j(t,0)=u^j(t,1)=0,~t\in[0,T],\\
&u^j(0,x)=0,~x\in[0,1],
\end{split}
\right.
\end{equation}
for $ j=1,\ldots,d$, 
where $T>0$, $u(t,x)=(u^1(t,x),\ldots,u^N(t,x))$, and $\{W^k\}_{k=1,\ldots,d}$ are $d$ independent Brownian sheets on some filtered probability space
$(\Omega,\mathscr F,\{\mathscr F_t\}_{t\ge 0}, \mathbb P)$. The hitting probabilities of the exact solution for system \eqref{heat} are well established by \cite{DKN07,BLX09}, which suggests that the critical dimensions in time and space directions are respectively $4$ and $2$, 
 that is for $A\in\mathscr B(\R^d)$,
\begin{gather*}
   c\operatorname{Cap}_{d-4}(A)  \le \uP\left\{u([T_0,T]\times \{x\})\cap A \neq \emptyset\right\}  \le C\mathscr{H}_{d-4}(A),\\
  c\operatorname{Cap}_{d-2}(A)  \le \uP\left\{u(\{t\}\times [\epsilon,1-\epsilon])\cap A \neq \emptyset\right\}  \le C\mathscr{H}_{d-2}(A).
 \end{gather*}
Here and after, $T_0\in(0,T)$ and $0<\epsilon\ll 1$ are two fixed numbers, and $c$, $C$ are generic positive constants that may differ from one place to another. The main result of this paper reveals that the critical dimensions of both temporal and spatial semi-discretizations are half of those of the exact solution. This indicates that for a large class of Borel sets $A$, the probability of the event that paths of the numerical solution hit $A$ cannot converge to that of the exact solution.

For the spatial semi-discretization of system \eqref{heat}, we introduce the finite difference method (FDM) and the spectral Galerkin method (SGM), and formulate
the corresponding numerical solutions $u^{j,N}(t,x)$ as stochastic integrals associated with discrete Green functions. We find that for any fixed space grid point $x_i=i/N\in[\epsilon,1-\epsilon]$, $U^N(t,x_i)=\big(u^{1,N}(t,x_i),\ldots,u^{d,N}(t,x_i)\big)$ given by FDM is H\"older continuous with respect to $t\in[T_0,T]$ with the optimal H\"older exponent $\frac{1}{2}$, which is crucial to conclude that the critical dimension associated to the time direction of FDM  is $2$. 
More precisely,  for $A\in\mathscr B(\R^d)$,
\begin{equation*}
c\operatorname{Cap}_{d-2}(A)  \le \uP\left\{U^N([T_0,T]\times\{x_i\})\cap A \neq \emptyset\right\}  \le C\mathscr{H}_{d-2}(A),
\end{equation*}
which can also be extended to the case of SGM.
 By noticing that $U^N(t,x)$ based on SGM is still a two-variable random field indexed by $(t,x)\in[0,T]\times[0,1]$, we further investigate its critical dimension in space direction. 
The main difficulty lies in  establishing lower bounds for H\"older continuity and the conditional variance in terms of the associated canonical metric $|x-y|$, which is overcome by refined estimates of the discrete Green function. These yield that the critical dimension associated to the space direction of SGM is $1$, i.e.,  for $A\in\mathscr B(\R^d)$,
$$c\operatorname{Cap}_{d-1}(A)  \le \uP\left\{U^N(\{t\}\times [\epsilon,1-\epsilon])\cap A \neq \emptyset\right\}  \le C \mathscr{H}_{d-1}(A).$$

For the temporal semi-discretization of system \eqref{heat}, we apply the exponential Euler method (EEM) with time stepsize $T/M$, $M\in\mathbb{N}_+$.
For any fixed time grid point $t_j=j\frac{T}{M}\in[T_0,T]$, the numerical solution $U_M(t_j,\cdot)$ of EEM is smoother than the exact solution  since  the temporal discretization avoids the treatment of the singularity of the Green function $G_t(x,y)$ near $t=0$. Actually,  making use of this property, we show that $U_M(t_j,x)$ is  Lipschitz continuous with respect to $x\in[\epsilon,1-\epsilon]$, and $1$ is exactly the optimal H\"older exponent. As a consequence,  the critical dimension associated to the space direction of EEM is also $1$:
\begin{equation*}
c\operatorname{Cap}_{d-1}(A)  \le \mathbb{P}\left\{U_M(\{t_j\}\times[\epsilon,1-\epsilon])\cap A \neq \emptyset\right\}  \le C\mathscr{H}_{d-1}(A),~\forall\,A\in\mathscr B(\R^d).
\end{equation*}
We remark that for the continuous EEM numerical solution $U_M(t,x)$, we can only obtain the 
upper bound of hitting probabilities in time direction in terms of Hausdorff measure since $U_M(t,x)$ is  smoother in every subinterval $(t_i, t_{i+1})$ than in grid points. 
It is worth mentioning that different from the infinite dimensional case, the continuous EEM numerical solution for the system of finite dimensional Ornstein--Ulenbeck equations preserves the critical dimension of the exact solution of the original system.

The rest of this paper is organized as follows.  Section \ref{S2} states some preliminaries, including the model, numerical discretizations and  the criteria on hitting probabilities for Gaussian random fields. Our main results on the upper and lower bounds for hitting probabilities of semi-discretizations are presented in Section \ref{S3}. Detailed proofs are postponed to Section \ref{S4}.

\section{Preliminaries}\label{S2}
In this section, we introduce the model  and its numerical discretizations, and  the criteria on hitting probabilities for Gaussian random fields.
 
\subsection{The model}
Consider the following linear stochastic  parabolic equation:
\begin{equation}\label{heatv0}
\left\{
\begin{split}
&\partial_tv(t,x)-\partial_{xx}v(t,x)=\dot{W}(t,x),\\
&v(t,0)=v(t,1)=0,~t\in[0,T],\\
&v(0,x)=0,~x\in[0,1],
\end{split}
\right.
\end{equation}
where
$W$ is a Brownian sheet on $[0,T]\times[0,1]$.
Then \begin{align}\label{Heatw}
v(t,x)=\int_0^t\int_0^1G_{t-r}(x,z)W(\ud r\ud z),~(t,x)\in[0,T]\times[0,1],
\end{align}
and the components $u^j,\,j=1,\ldots,d$, in \eqref{heat} are independent copies of $v$.
Here, the Green function $G$  has the expression
\begin{align}\label{GD0}
G_{t}(x,y)=\sum_{k=1}^\infty e^{-\pi^2 k^2t}e_k(x)e_k(y)
\end{align}
with $e_k(x)=\sqrt{2}\sin(k\pi x)$, $k\ge 1$.
 We recall 
another equivalent formulation
 \begin{align*}
 G_{t}(x, y)=
      \frac{1}{\sqrt{4 \pi t}} \sum_{n=-\infty}^{+\infty}\left(e^{-\frac{(x-y-2 n)^{2}}{4 t}}-e^{-\frac{(x+y-2 n)^{2}}{4 t}}\right),
       \end{align*}  
 from which, one can observe that
  \begin{align}\label{GHP}
  G_t(x,y)=P_t(x,y)+H_t(x,y)
  \end{align}
   with $P_t(x,y)=\frac{1}{\sqrt{4\pi t}}e^{-\frac{(x-y)^{2}}{4t}}$ being the heat kernel on $\R$ and $H_t(x,y)$ being a smooth function on $[0,T]\times(0,1)^2$ (see e.g., \cite[Corollary 3.4]{WJ86}). We also recall that for any $0<\epsilon\ll 1$,
   \begin{equation}\label{PGP}
     \frac{1}{C}P_t(x,y)\le G_t(x,y)\le CP_t(x,y),~\forall\,x,y\in [\epsilon,1-\epsilon],\,t\in[T_0,T],
   \end{equation}
   for some $C:=C(T,\epsilon)>0$,
and refer  to \cite[Appendix]{BP98} for more properties on $G$.
\subsection{Numerical methods}
In this part, we give some spatial and temporal discretizations  for the linear stochastic  parabolic system \eqref{heat}.

\subsubsection{Spatial discretizations}

For spatial discretizations of \eqref{heat}, we introduce the spectral Galerkin method and the finite difference method. Their numerical solutions can be written as
\begin{equation}\label{Heatw1}
v^N(t,x)=\int_0^t\int_0^{1}G^N_{t-r}(x,z)W(\ud r\ud z),~(t,x)\in (0,T]\times[0,1],
\end{equation}  
where $G^N$ is given in \eqref{NGDG} for SGM and in \eqref{NGDF} for FDM.
The associated numerical solution to \eqref{heat} is denoted by
\begin{equation}\label{UNTX}
U^{N}(t,x)=\big(u^{1,N}(t,x),\ldots,u^{d,N}(t,x)\big),
\end{equation}
where $u^{j,N}(t,x)$ are generated via replacing $W$  in \eqref{Heatw1} by $W^j,\,j=1,\ldots,d$.

\textit{Spectral Galerkin method:} 
 Recalling \eqref{heatv0}, we denote $v(t)=v(t,\cdot)\in H:=L^2(0,1)$.
  Then $v(t)$ satisfies the following infinite dimensional evolution equation
\begin{equation*}
\ud v(t)=\Delta v(t)\ud t+\ud V_t,~v(0)=0
\end{equation*}
in distribution sense,  where $\Delta$ is the Dirichlet Laplacian, and $V_t=\sum_{k=1}^\infty\beta_k(t)e_k$ is a cylindrical Wiener process on $H$. Here, $\{\beta_k(t),t\ge0\}_{k\ge1}$ is a sequence of independent 
 Brownian motions on $(\Omega,\mathscr F,\{\mathscr F_t\}_{t\ge 0}, \mathbb P)$.

 For $N\ge 1$, define $\mathbb{Z}_N=\{1,\ldots,N-1\}$.
We introduce the finite dimensional subspace $H_N:=\textrm{span}\{e_k,\,k\in\mathbb{Z}_N\}$ of $H$  and the projection operator $P_N: H\rightarrow H_N$ given by $P_Nh=\sum_{k=1}^{N-1}\langle e_k,h\rangle_He_k,\,h\in H$. Then 
the spectral Galerkin approximation (see e.g, \cite{CHL17,CHS21}) for \eqref{heatv0} in  $H_N$ is formulated as
$$\ud v_t^N=P_N\Delta v_t^N\ud t+P_N\ud V_t,~v_0^N=0,$$
which is equivalent to a system of stochastic differential equations:
$$\ud\langle v_t^N,e_k\rangle_H=-k^2\pi^2\langle v_t^N,e_k\rangle_H\ud t+\ud\beta_k(t),~\langle v_0^N,e_k\rangle_H=0,~k\in\mathbb{Z}_N.$$
Noticing $\sum_{k=1}^{N-1}\langle v^N_t,e_k\rangle_H  e_k(x)=v^N(t,x)$ in distribution sense, one can  verify that 
\begin{equation*}
v^N(t,x)=\int_0^t\int_0^{1}G^N_{t-r}(x,z)W(\ud r\ud z),~(t,x)\in (0,T]\times[0,1],
\end{equation*}  
where
the discrete heat kernel associated to SGM is 
\begin{align}\label{NGDG}
G^N_t(x,y)=\sum_{k=1}^{N-1}\exp(-k^2\pi^2t)e_k(x)e_k(y).
\end{align}

\textit{Finite difference method:}
Using the central difference, the finite difference method of \eqref{heatv0} is proposed in \cite{GI98}. The associated numerical solution is 
\begin{equation*}
v^N(t,x_i)=\int_0^t\int_0^{1}\sum_{k=1}^{N-1}\exp(\lambda_k^Nt)e_k(x_i)e_k(\kappa_N(z))W(\ud r\ud z),~(t,x)\in (0,T]\times[0,1],
\end{equation*}  
where $x_i=\frac{i}{N}$, $\lambda_k^N=-4N^2\sin^2\left(\frac{k}{2N}\pi\right)$ and $\kappa_N(z):=\frac{[Nz]}{N}$. We remark that $-k^2\pi^2\le\lambda_k^N\le-4k^2,$ $\forall\,k\in\mathbb{Z}_N$. 
By the linear interpolation, we obtain the continuous numerical solution  \eqref{Heatw1} of FDM,
where the associated discrete Green function is
\begin{align}\label{NGDF}
G^N_t(x,y)=\sum_{k=1}^{N-1}\exp(\lambda_k^Nt)e^N_k(x)e_k(\kappa_N(y))
\end{align}
with
\begin{equation}\label{sinN0}
e_k^N(x):=e_k\left(\kappa_N(x)\right)+N(x-\kappa_N(x))\left[e_k\left(\kappa_N(x)+\frac{1}{N}\right)-e_k\left(\kappa_N(x)\right)\right],~x\in[0,1].
\end{equation}

\subsubsection{Temporal discretization}
Introducing a time stepsize $\delta t=\frac{T}{M},$ $M\in \mathbb N_+$, the numerical solution of the exponential Euler method for \eqref{heatv0} is given by
\begin{align*}
v_M(t_{i},x)=\int_0^1G_{\delta t}(x,z)v_M(t_{i-1},z)\ud z+\int_{t_{i-1}}^{t_{i}}\int_0^1G_{\delta t}(x,z)W(\ud z\ud r),
\end{align*}
where $t_i=i\delta t$, $i\in\mathbb{Z}_{M+1}=\{1,\ldots,M\}$.
Hence for $i\in\mathbb{Z}_{M+1}$,
\begin{equation}\label{HeatM}
v_M(t_{i},x)=\int_0^{t_{i}}\int_0^1G_{t_i-[\frac{r}{\delta t}]\delta t}(x,z)W(\ud z\ud r).
\end{equation}
The associated numerical solution of system \eqref{heat} is denoted by 
\begin{equation}\label{UMTX}
U_{M}(t_i,x)=\big(u_{M}^1(t_i,x),\ldots,u_{M}^d(t_i,x)\big),~i\in\mathbb{Z}_{M+1},
\end{equation}
 where $u_{M}^j(t_i,x)$ is generated by replacing $W$ in \eqref{HeatM} by $W^j$ for $j\in\{1,\ldots,d\}$.


 \subsection{Hitting probability}\label{S2.3}

Given two random variables $X$ and $Y$, we denote $\Var X:=\E|X-\E X|^2$ and 
$\Cov(X,Y):=\E[(X-\E X)(Y-\E Y)]$. 
For any Borel set  $F \subset \R^{d}$,  define  $\mathcal P(F)$  to be the set of all probability measures with compact support in  $F$.  For $\mu \in \mathcal P(\R^{d})$,  let  $I_{\beta}(\mu)$  denote the  $\beta$-dimensional energy of  $\mu$, i.e.,
$$I_{\beta}(\mu):=\iint \mathrm{K}_{\beta}(\|x-y\|) \mu(\ud x) \mu(\ud y),$$
where  $\|x\|$  denotes the Euclidean norm of  $x \in \R^{d}$, and
\begin{equation*}
\mathrm{K}_{\beta}(r):=\left\{\begin{array}{ll}
r^{-\beta} & \text { if } \beta>0, \\
\log \left(\frac{e}{r\wedge1}\right) & \text { if } \beta=0, \\
1 & \text { if } \beta<0.
\end{array}\right.
\end{equation*}
For any  $\beta \in\R$ and Borel set  $F \subset\R^{d}$, 
$\operatorname{Cap}_{\beta}(F)$  denotes the  $\beta$-dimensional Bessel--Riesz capacity of  $F$, that is,
$$\operatorname{Cap}_{\beta}(F):=\left[\inf _{\mu \in \mathcal{P}(F)} I_{\beta}(\mu)\right]^{-1},$$ where $1/\infty:=0.$ Given  $\beta \ge 0$,  the  $\beta$-dimensional Hausdorff measure of  $F$  is defined by
$$\mathscr{H}_{\beta}(F)=\lim _{\epsilon \rightarrow 0^{+}} \inf \left\{\sum_{i=1}^{\infty}\left(2 r_{i}\right)^{\beta}: F \subset \bigcup_{i=1}^{\infty} B\left(x_{i}, r_{i}\right), \sup _{i \geq 1} r_{i} \leq \epsilon\right\},$$
where  $B(x, r)$  denotes the open Euclidean ball of radius  $r>0$  centered at  $x \in \R^{d}$.  When  $\beta<0$,   $\mathscr{H}_{\beta}(F)$ is defined to be infinite.

 Based on the Bessel-Riesz capacity and Hausdorff measure, we present the lower and upper bounds for hitting probabilities of the exact solution of \eqref{heat}.  

 \begin{tho}\cite[Theorem 4.6]{DKN07}\label{heat-hit}
 Let $u(t,x)=(u^1(t,x),\ldots,u^d(t,x))$ be the exact solution of \eqref{heat}. Fix $L>0$, $T_0\in(0,T)$, and $0<\epsilon\ll 1$.
Then for any sufficiently small $\delta>0$,
there exist $C_i=C_i(\delta,T_0,T,d,\epsilon,L),\,i=1,\ldots,6$ such that for each compact set $A\subset[-L,L]^d$,
\begin{align*}
&C_1\operatorname{Cap}_{d-6}(A)  \le \uP\left\{u([T_0,T]\times [\epsilon,1-\epsilon])\cap A \neq \emptyset\right\}  \le C_2 \mathscr{H}_{d-6-\delta}(A),\\
 & C_3\operatorname{Cap}_{d-2}(A)  \le \uP\left\{u(\{t\}\times [\epsilon,1-\epsilon])\cap A \neq \emptyset\right\}  \le C_4 \mathscr{H}_{d-2-\delta}(A),\\
 &  C_5\operatorname{Cap}_{d-4}(A)  \le \uP\left\{u([T_0,T]\times \{x\})\cap A \neq \emptyset\right\}  \le C_6 \mathscr{H}_{d-4-\delta}(A),
 \end{align*}
 where $t\in[T_0,T]$ and $x\in[\epsilon,1-\epsilon]$.
 \end{tho}

 The  infinitesimal factor $\delta$ in Theorem \ref{heat-hit} can be removed by applying the following criteria about hitting probabilities of a general Gaussian random field, which also can  be applied to semi-discretizations of \eqref{heat}.

\begin{tho}\cite[Theorem 2.1]{BLX09}\label{Gausshit}
Let $I=[a,b]:=\prod_{j=1}^m[a_j,b_j]$ $(a_j<b_j)$ be an interval or a rectangle in $\R^m$ and $X=\left\{X(x),\,x \in \mathbb{R}^{m}\right\}$  be an $\R^d$-valued Gaussian random field with coordinate processes  $X_{1}, \ldots, X_{d}$  being independent copies of a real-valued, centered Gaussian random field  $X_{0}=\left\{X_{0}(x),\, x \in \mathbb{R}^{m}\right\}$. Assume that the following conditions hold:


$\mathsf{(C0)}$ for all  $x \in I$, $\mathbb{E}|X_{0}(x)|^{2}\ge c_{1}$;

$\mathsf{(C1)}$ 
there exists $H=(H_1,\ldots,H_m)\in(0,1)^m$  such that for all $x, y \in I,$
$$c_{2} \sum_{j=1}^{m}\left|x_{j}-y_{j}\right|^{2 H_{j}}  \le \mathbb{E}\left|X_{0}(x)-X_{0}(y)\right|^{2} \le c_{3} \sum_{j=1}^{m}\left|x_{j}-y_{j}\right|^{2 H_{j}};$$

$\mathsf{(C2)}$  for all  $x,y \in I,$ 
\begin{equation*}\label{C2}
\operatorname{Var}\left(X_{0}(x) | X_{0}(y)\right) \ge c_{4} \sum_{j=1}^{m}\left|x_{j}-y_{j}\right|^{2 H_{j}}.
\end{equation*}
Here, $c_i,~i=1,2,3,4$, are positive constants independent of $x,y\in I$, and $\operatorname{Var}\left(X_{0}(x) | X_{0}(y)\right) $ denotes the conditional variance of  $X_{0}(x)$  given  $X_{0}(y)$. Then there exist positive constants $c_5$, $c_6$ such that for every Borel set $A$ in $\mathbb{R}^{d}$,
\begin{equation}\label{Cahu}
c_{5} \operatorname{Cap}_{d-Q}(A)  \le \mathbb{P}\left\{X(I)\cap A \neq \emptyset\right\}  \le c_{6} \mathscr{H}_{d-Q}(A),
\end{equation}
where $Q:=   \sum_{j=1}^{m} 1 / H_{j}$.
\end{tho}

  For different kinds of random fields, the most concerned issue is the value of $Q$.    If a random field $X$ satisfies \eqref{Cahu} for some integer $Q$, it is well known that $Q$ is the \textit{critical dimension} for hitting points, which means that points are polar for $X$ when $d> Q$, and are non-polar for $X$ when $d< Q$. We would like to mention that for Gaussian random fields, the upper bound in $\mathsf{(C1)} $ suffices to 
 derive the upper bound in \eqref{Cahu}, and the lower bounds in $\mathsf{(C1)}$ and $\mathsf{(C2)}$  are used to deduce the lower bound in \eqref{Cahu}.

\begin{rem}\label{rem1}
As stated in \cite[Section 2]{BLX09},  $\mathsf{(C1)}$ and $\mathsf{(C2)}$ are closely related.  For the case of $H\in(0,1)^m$, if the function $x\mapsto\E|X_0(x)|^2$ satisfies a certain smoothness condition, for instance, it has continuous first-order partial derivatives on $I$, then one can show that $\mathsf{(C1)}$ implies $\mathsf{(C2)}$  by using the following fact:  
 if $(Y,Z)$  is a centered Gaussian vector, then
\begin{align}\label{VarYZ}
\operatorname{Var}(Y| Z)
=\frac{\left(\rho_{Y,Z}^{2}-(\sigma_{Y}-\sigma_{Z})^{2}\right)\left((\sigma_{Y}+\sigma_{Z})^{2}-\rho_{Y,Z}^{2}\right)}{4 \sigma_{Z}^{2}},
\end{align}
 where $\rho_{Y,Z}^{2}=\mathbb{E}|Y-Z|^{2},\, \sigma_{Y}^{2}=\mathbb{E}[Y^{2}]$ and  $\sigma_{Z}^{2}=\mathbb{E}[Z^{2}]$.
 In fact, the above smoothness condition can also be weakened to be that
the function $x\mapsto\E|X_0(x)|^2$ is H\"older continuous with exponent $(H_1(1+\eta),\ldots,H_m(1+\eta))$ for some $\eta>0$.

%
%
\end{rem}

\begin{rem}\label{rm2.4}
For the work on hitting probabilities of $X$ with optimal H\"older continuity exponent $H=1$,
we are only aware of \cite{SV18}, where $X_0$ is the exact solution of the linear stochastic Poisson system on $[0,b],$ $b>0$, driven by additive white noises. 
  From the proofs of \cite[Theorem 5.10 and Theorem 5.11 ($k=1$)]{SV18}, one can see that
for the special case of $m=1$ and $H=1$, in order to obtain \eqref{Cahu},  it suffices to show that $X_0$ satisfies  $\mathsf{(C0)}$, $\mathsf{(C1)}$ and $\mathsf{(C2)}$$^\prime$:

$\mathsf{(C2)}$$^\prime$ there exist  positive constants  $c_{4},C_4$  such that for all  $x,\, y \in I,$ 
\begin{equation*}\label{C2'}
c_4\left|x-y\right|^2\le \operatorname{Var}\left(X_{0}(x) | X_{0}(y)\right) \le C_4\left|x-y\right|^2.
\end{equation*}
Here, $\mathsf{(C2)}$$^\prime$  corresponds to formula (77) in \cite{SV18}.

Noticing that by  Lemma \ref{remH=1}, $\mathsf{(C0)}$ and $\mathsf{(C1)}$ imply the right side of $\mathsf{(C2)}$$^\prime$, hence
 Theorem \ref{Gausshit} is also applicable to the case of $m=1$ and $H=1$. However,
for the case of $m=1$ and  $H=1$, the verification of $\mathsf{(C2)}$ is generally more involved since we cannot expect that  the function $x\mapsto\E|X_0(x)|^2$ is $(1+\eta)$-H\"older continuous with some $\eta>0$ (see Remark \ref{rem1}). 
\end{rem}

\begin{lem}\label{remH=1}
Let $m=1$, $H=1$. If $\mathsf{(C0)}$ and $\mathsf{(C1)}$ hold, then we have
\begin{equation*}
 \operatorname{Var}\left(X_{0}(x) | X_{0}(y)\right) \le C_4\left|x-y\right|^2,~\forall\,x,y\in I.
\end{equation*}
\end{lem}

\begin{proof}
Due to \eqref{VarYZ}, it suffices to show that there is $C>0$ such that for any $x,y\in I$,
\begin{align}\label{H=11}
\E|X_0(x)-X_0(y)|^2-\left(\sqrt{\E|X_0(x))^2}-\sqrt{\E|X_0(y)|^2}\right)^2\le C|x-y|^2
\end{align}
and
\begin{align}\label{H=12}
\left(\sqrt{\E|X_0(x))^2}+\sqrt{\E|X_0(y)|^2}\right)^2-\E|X_0(x)-X_0(y)|^2 \le C,
\end{align}
since H\"older's inequality implies that the left sides of  both \eqref{H=11} and \eqref{H=12} are always nonnegative. From $\mathsf{(C0)}$ and $\mathsf{(C1)}$, \eqref{H=11} and \eqref{H=12} follow immediately and the proof is completed .
\end{proof}

\section{Main results}\label{S3}
 The main results of this paper are the following Theorems \ref{finite-space-hit}, \ref{spectral-space-hit}, and \ref{EE-hit}, which can be summarized as follows:  the critical dimensions of both temporal and spatial semi-discretizations are half of those of the exact solution. As a consequence, for a large class of Borel sets $A$, the probability of the event that the paths of the numerical solution hit $A$ cannot converge to that of the exact solution.

 \subsection{Hitting probabilities of spatial and temporal discretizations}
 
 In this part, we state our main results on hitting probabilities of spatial and temporal discretizations, whose proofs are postponed to Section \ref{S4}. We first give the  hitting probabilities in time direction
 for the spatial semi-discretization of FDM.
\begin{tho}\label{finite-space-hit}
Let $N\gg1$ and $U^N(t,x)$ defined in \eqref{UNTX} be the numerical solution of FDM for  system \eqref{heat}. 
 Then there are $C_i=C_i(N,\epsilon,T_0,T,d),\,i=1,2$ such that for any Borel set $A$ in $\R^d$,
\begin{equation*}
C_1\operatorname{Cap}_{d-2}(A)  \le \mathbb{P}\left\{U^N([T_0,T]\times\{x\})\cap A \neq \emptyset\right\}  \le C_2\mathscr{H}_{d-2}(A),
\end{equation*}
where $x\in[\epsilon,1-\epsilon]\cap\{\frac{1}{N},\ldots,\frac{N-1}{N}\}$.
\end{tho}


Theorem \ref{finite-space-hit} can be extended to the case of SGM. In addition, the numerical solution based on SGM is still a continuous Gaussian random field indexed by $(t,x)\in[0,T]\times[0,1]$, hence we further investigate its hitting probabilities in space direction.

 \begin{tho}\label{spectral-space-hit}
 Let $U^N(t,x)$  defined in \eqref{UNTX} be the numerical solution of SGM for  system \eqref{heat}. Then for sufficiently large $N$, there exist $C_i=C_i(N,\epsilon,T_0,T,d),\,i=1,2,3,4$ such that for
 all Borel set $A$ in $\R^d$,
\begin{gather}\label{exas3}
  C_1\operatorname{Cap}_{d-2}(A)  \le \uP\left\{U^N([T_0,T]\times\{x\})\cap A \neq \emptyset\right\}  \le C_2 \mathscr{H}_{d-2}(A),\\
  \label{exas1}
    C_3\operatorname{Cap}_{d-1}(A)  \le \uP\left\{U^N(\{t\}\times [\epsilon,1-\epsilon])\cap A \neq \emptyset\right\}  \le C_4 \mathscr{H}_{d-1}(A),
 \end{gather}
 where $x\in[\epsilon,1-\epsilon]$ and $t\in[T_0,T]$.
 \end{tho}


For the temporal semi-discretization based on EEM, we study its hitting probabilities in space direction.
\begin{tho}\label{EE-hit}
Let $U_M(t,x)$  defined in \eqref{UMTX} be the numerical solution of EEM for  system \eqref{heat}. 
Then there are $C_i=C_i(M,\epsilon,T_0,T,d),\,i=1,2$ such that for any Borel set $A$ in $\R^d$,
\begin{equation}\label{Cahutd}
C_1\operatorname{Cap}_{d-1}(A)  \le \mathbb{P}\left\{U_M(\{t\}\times[\epsilon,1-\epsilon])\cap A \neq \emptyset\right\}  \le C_2\mathscr{H}_{d-1}(A),
\end{equation}
where $t\in\{\frac{1}{M},\ldots,1\}\cap[T_0,T]$.
\end{tho}

\newpage

Theorems \ref{finite-space-hit}, \ref{spectral-space-hit} and \ref{EE-hit} reveal that for some Borel sets $A$, the probability of the event that paths of the numerical solution hit $A$ cannot converge to that of the exact solution. More precisely,
by Frostman's theorem (\cite[Appendix C, Theorem 2.2.1]{KD02}), for any compact set $A\subset\R^d$,
\begin{align*}
\operatorname{dim}_\mathrm{H}(A)=&\sup \left\{s>0: \mathscr{H}_{s}(A)=\infty\right\}=\inf \left\{s>0: \mathscr{H}_{s}(A)=0\right\}\\
=&\sup \left\{s>0:  \operatorname{Cap}_s(A)>0\right\}=\inf \left\{s>0: \operatorname{Cap}_{s}(A)=0\right\},
\end{align*}
where $\operatorname{dim}_\mathrm{H}(A)$ is the Hausdorff dimension of $A$. Therefore, 
$$\operatorname{Cap}_{d-Q_\text{exact}}(A)>0, ~\forall\, d<\operatorname{dim}_\mathrm{H}(A)+Q_\text{exact},$$ and$$\mathscr{H}_{d-\frac{1}{2}Q_\text{exact}}(A)=0,~\forall\,d>\operatorname{dim}_\mathrm{H}(A)+\frac{1}{2}Q_\text{exact}.$$
Theorem \ref{heat-hit}  shows that for the exact solution $u$ of system \eqref{heat}, the critical dimension $Q^t_\text{exact}=4$ in time direction and  the critical dimension  $Q^x_\text{exact}=2$ in space direction. Theorems \ref{finite-space-hit}, \ref{spectral-space-hit} and \ref{EE-hit} indicate that for any $A\subset\R^d$ with $d\in\left(\operatorname{dim}_\mathrm{H}(A)+\frac{1}{2}Q^t_\text{exact},\,\operatorname{dim}_\mathrm{H}(A)+Q^t_\text{exact}\right)$, 
 $$\lim_{N\rightarrow \infty} \mathbb{P}\left\{U^N([T_0,T]\times\{x\})\cap A \neq \emptyset\right\}=0<   \mathbb{P}\left\{u([T_0,T]\times\{x\})\cap A \neq \emptyset\right\},$$
 and
for any $A\subset\R^d$ with $d\in\left(\operatorname{dim}_\mathrm{H}(A)+\frac{1}{2}Q^x_\text{exact},\,\operatorname{dim}_\mathrm{H}(A)+Q^x_\text{exact}\right)$, 
 $$\lim_{M\rightarrow \infty} \mathbb{P}\left\{U_M(\{t\}\times[\epsilon,1-\epsilon])\cap A \neq \emptyset\right\}=0<   \mathbb{P}\left\{u(\{t\}\times[\epsilon,1-\epsilon])\cap A \neq \emptyset\right\}.$$
For example, for $d=3$ and each $y\in\R^3$, $\dim_\textrm{H}(\{y\})=0$ (see e.g., \cite[Example 2.2]{AM18}), and hence $d\in(\operatorname{dim}_\mathrm{H}(\{y\})+\frac{1}{2}Q^t_\text{exact},\operatorname{dim}_\mathrm{H}(\{y\})+Q^t_\text{exact})$. This implies that
 for fixed $x\in[\epsilon,1-\epsilon]$, all points $y\in\R^3$ are nonpolar for $u(\cdot,x)$ but polar for the spatial semi-discretization $U^N(\cdot,x)$.

\subsection{Comparison with the finite dimensional situation}\label{SEC3.2}

Interpolation is usually used to extend the numerical solution from grid points to the whole interval. In view of \eqref{HeatM}, it is natural to define
the continuous exponential Euler numerical solution 
by 
\begin{equation}\label{Heatcon}
v_M(t,x)=\int_0^{t}\int_0^1G_{t-[\frac{r}{\delta t}]\delta t}(x,z)W(\ud z\ud r),
\end{equation}
where $[\cdot]$ denotes the greatest-integer function.
In the same way, we obtain the continuous EEM numerical solution $U_M(t,x)$ of system \eqref{heat}. 
We first study the  H\"older continuity of $v_M(t,x)$ in time, which is crucial to the analysis of hitting probabilities of $U_M(t,x)$.
\begin{lem}\label{Holder-h1Nt}
Let $v_M$ given by \eqref{HeatM} be the numerical solution  of EEM for \eqref{heatv0}. Then there exist positive constants $c_i=c_i(T_0,T,\epsilon),\,i=1,2$ such that for any $1\le j<i\le M$, 
\begin{align}\label{upphM}
c_1\sqrt{t_i-t_j}\le \E|v_M(t_i,x)-v_M(t_j,x)|^2\le c_2\sqrt{t_i-t_j},
\end{align}
where $x\in[\epsilon,1-\epsilon]$.
\end{lem}

 The following corollary indicates that $\frac{1}{4}$ is the upper bound of the H\"older exponent of $v_M(\cdot,x)$ but is not the optimal H\"older exponent. 

\begin{cor}\label{cor1}
Let the condition of Lemma \ref{Holder-h1Nt} hold and fix $x\in[\epsilon,1-\epsilon]$. Then there exists some positive constant $c_3=c_3(T_0,T)$ such that for any $T_0\le s< t\le T$, 
\begin{align}\label{vmst0}
\E|v_M(t,x)-v_M(s,x)|^2\le c_3\sqrt{t-s}.
\end{align}
However, there is no $c_4>0$ such that for any $T_0\le s< t\le T$,
\begin{align}\label{vmst1}
\E|v_M(t,x)-v_M(s,x)|^2\ge c_4\sqrt{t-s}.
\end{align}
\end{cor}

By  \eqref{vmst0}, we have that for any $x\in[\epsilon,1-\epsilon]$,
\begin{equation*}
\mathbb{P}\left\{U_M([T_0,T]\times\{x\})\cap A \neq \emptyset\right\}  \le C_2\mathscr{H}_{d-2}(A).
\end{equation*}
However, \eqref{vmst1} prevents us from deriving the lower bound of hitting probabilities in time direction of $U_M(t, x)$ in terms of Bessel-Riesz capacity. For infinite dimensional  stochastic differential equation, the continuous temporal semi-discretization numerical solution is smoother in every subinterval $(t_i,t_{i+1})$ than the exact solution. However, for the finite dimensional stochastic differential equation, the result is different.

Let $\{B(t)=(B^0(t),B^1(t),\ldots,B^d(t)),t\ge0\}$ be a standard $(d+1)$-dimensional Brownian motion on $(\Omega,\mathscr F,\{\mathscr F_t\}_{t\ge 0},\mathbb P)$, and $Y(t)=(Y^1(t),\ldots,Y^d(t))$ be the solution of  the following system
\begin{equation}\label{Yt}
\ud Y^i(t)=-\lambda Y^i(t)\ud t+\ud B^i(t),~t\in(0,T],~i=1,\ldots,d,
\end{equation}
where $\lambda>0$ and $Y^i(0)=0$. Obviously, each component $Y^i(t)$ is  an independent copy of  the 1-dimensional Ornstein--Ulenbeck process $\{Y^0(t),t\ge0\}$ which satisfies
 \begin{equation}\label{Yt0}
 \ud Y^0(t)=-\lambda Y^0(t)\ud t+\ud B^0(t),~ t\in(0,T]; ~ Y^0(0)=0.
 \end{equation} 
 Obviously, we have 
  \begin{equation}\label{Yt00}
  Y^0(t)=\int_0^te^{-\lambda(t-r)}\ud B^0(r), ~t\ge0,
 \end{equation} 
from which, one obtains that for $s<t$,
  \begin{align*}
  \E|Y^0(t)-Y^0(s)|^2
  =&\int_s^te^{-2\lambda(t-r)}\ud r+\int_0^s|e^{-\lambda(t-r)}-e^{-\lambda(s-r)}|^2\ud r.
 \end{align*}
 
It is clear that $  \E|Y^0(t)-Y^0(s)|^2\le C(T,\lambda)|t-s|$ and 
$$\E|Y^0(t)-Y^0(s)|^2\ge\int_s^te^{-2\lambda(t-r)}\ud r\ge e^{-2\lambda T}|t-s|$$ 
 for all $0<s<t\le T$.
Besides, $\E|Y^0_t|^2=\frac{1-e^{-2\lambda t}}{2\lambda}\ge \frac{1-e^{-2\lambda t_0}}{2\lambda}$ for all $t\in[T_0,T]$, and $\E|Y^0_\cdot|^2$ is a Lipschitz continuous function on $[0,T]$. By Theorem \ref{Gausshit} and Remark \ref{rem1},  we deduce that for every Borel set $A$ in $\R^d$,
\begin{equation*}
C_{1} \operatorname{Cap}_{d-2}(A)  \le \mathbb{P}\left\{Y([T_0,T])\cap A \neq \emptyset\right\}  \le C_{2} \mathscr{H}_{d-2}(A)
\end{equation*}
with $C_1,C_2$ being positive constants depending on $T_0,T,d,\lambda$.

  When we apply EEM to discretize \eqref{Yt0} and use the same continuous approach as in \eqref{Heatcon},  the associated numerical solution is
 \begin{equation*}
\bar Y^0(t)=\int_0^{t}e^{-\lambda(t-[\frac{r}{\delta t}]\delta t)}\ud B^0(r),~t\in(0,T],
\end{equation*}
which has a similar formulation as in \eqref{Yt00}. Analogous estimates yield that
\begin{equation*}
C_1 \operatorname{Cap}_{d-2}(A)  \le \mathbb{P}\left\{\bar Y([T_0,T])\cap A \neq \emptyset\right\}  \le C_2 \mathscr{H}_{d-2}(A)
\end{equation*}
for every Borel set $A$ in $\R^d$, where $\bar Y$ is the continuous exponential Euler approximation of $Y$. It can be concluded that the continuous exponential Euler numerical solution $\bar Y=\{\bar Y(t),t\in[T_0,T]\}$ for system
\eqref{Yt} preserves the critical dimension of  the exact solution $Y=\{Y(t),t\in[T_0,T]\}$, which is different from the infinite dimensional case. In fact, this property not only holds for the continuous exponential Euler numerical solution, but also holds for the  Euler--Maruyama method under a proper continuity approach.

The Euler--Maruyama method applied to \eqref{Yt0} yields
$$Y_{i}^0=Y_{i-1}^0-\lambda\delta t Y_{i-1}^0+\triangle B_{i-1}^0,~ i\in\mathbb{Z}_{M+1},$$
where $\triangle B_i^0=B^0(t_{i+1})-B^0(t_i)$. After rearranging, we have
$$Y_{i}^0=\sum_{k=0}^{i-1}(1-\lambda\delta t)^{i-1-k}\triangle B_k^0=\int_0^{t_i}(1-\lambda\delta t)^{[\frac{t_i-r}{\delta t}]}\ud B^0(r),~ i\in\mathbb{Z}_{M+1}.$$
Naturally, we define the continuous Euler--Maruyama numerical  solution for \eqref{Yt0} by
\begin{equation}\label{YOt}
\widetilde Y^0(t)=\int_0^{t}(1-\lambda\delta t)^{[\frac{t-r}{\delta t}]}\ud B^0(r),~ i\in\mathbb{Z}_{M+1}.
\end{equation}
\begin{prop}\label{OUn}
Fix $\delta t\in(0,\frac{1}{\lambda})$ and 
let $\widetilde Y(t)=(\widetilde Y^1(t),\ldots,\widetilde Y^d(t))$ be the  continuous Euler--Maruyama numerical solution for system \eqref{Yt}. Then there exist positive constants $C_i=C_i(T_0,T,d,\lambda),$ $i=1,2$ such that for every Borel set $A$ in $\R^d$,
\begin{equation*}
C_1 \operatorname{Cap}_{d-2}(A)  \le \mathbb{P}\left\{\widetilde Y([T_0,T])\cap A \neq \emptyset\right\}  \le C_2 \mathscr{H}_{d-2}(A).
\end{equation*}
\end{prop}

In general, the hitting probabilities of continuous versions of numerical solutions depend on continuous approaches.
If consider the linear interpolation of the Euler--Maruyama numerical solution $\{Y_i^0,\,i\in\mathbb{Z}_{M+1}\}$,
 $$\widetilde Y^0(t)=\frac{t_{i}-t}{\delta t}\widetilde Y^0_{i-1}+\frac{t-t_{i-1}}{\delta t}\widetilde Y^0_{i},~ t\in(t_{i-1},t_{i}),~ i\in\mathbb{Z}_{M+1},$$
then we cannot obtain that there is $c>0$ such that
 \begin{align*}
  \E|\widetilde Y^0(t)-\widetilde Y^0(s)|^2\ge c(t-s),~\,\forall \,T\ge t>s\ge T_0.
 \end{align*}
 Actually, for any $t_m<s<t<t_{m+1}$, 
 \begin{align*}
  \E|\widetilde Y^0(t)-\widetilde Y^0(s)|^2=\frac{(t-s)^2}{(\delta t)^2}  \E|\widetilde Y^0(t_{i+1})-\widetilde Y^0(t_i)|^2\le C(\delta t,T,\lambda)\frac{(t-s)^2}{\delta t}.
  \end{align*}
Therefore, the linear interpolation is not a proper choice to inherit the critical dimension of the exact solution.

\section{Proofs}\label{S4}

In this section, we present the proofs of main results in Section \ref{S3}.
\subsection{Proof of Theorem \ref{finite-space-hit}}

Based on Theorem \ref{Gausshit}, we only need to prove that the  numerical solution $v^N$ of the spatial discretization of FDM satisfies $\mathsf{(C0)}$-$\mathsf{(C2)}$.
The following lemma is prepared for deriving the optimal H\"older continuity of $v^N(t,x)$ with respect to $x\in[\epsilon,1-\epsilon]$. 
\begin{lem}\label{sincos}
Let $0<\epsilon\ll1$ and $N>8$. Then for any $x,y\in[\epsilon,1-\epsilon]$,
\begin{align}\label{sin}
&|e_1(x)-e_1(y)|^2+|e_2(x)-e_2(y)|^2\ge c(\epsilon)|x-y|^2,\\\label{sinN}
&|e^N_1(x)-e^N_1(y)|^2+|e^N_2(x)-e^N_2(y)|^2\ge c(\epsilon,N)|x-y|^2,
\end{align}
where $c(\epsilon)$ and $c(\epsilon,N)$ are positive constants.
\end{lem}
\begin{proof}
Let $x,y\in[\epsilon,1-\epsilon]$. The proof of \eqref{sin} is divided into two cases.

\textit{Case 1:} $x+y\in [2\epsilon,1-\epsilon]\cup[1+\epsilon,2-2\epsilon]$. In this case, $2\epsilon-1\le x-y\le 1-2\epsilon$. Therefore,
\begin{align*}
|\sin(\pi x)-\sin(\pi y)|^2
=4\cos^2\left(\frac{\pi (x+y)}{2}\right)\sin^2\left(\frac{\pi (x-y)}{2}\right)\ge4\sin^2\left(\frac{\pi \epsilon}{2}\right)|x-y|^2,
\end{align*}
because $\frac{2}{\pi}\le\frac{\sin\theta}{\theta}\le 1$ holds for all $\theta\in(0,\frac{\pi}{2}].$

\textit{Case 2:} $x+y\in [1-\epsilon,1+\epsilon]$. In this case, we have $|\cos\left(\pi (x+y)\right)|\ge \cos(\pi\epsilon)$ and $-(1-\epsilon)\le x-y\le 1-\epsilon$. This implies that
$|\sin\left(\pi (x-y)\right)|\ge\frac{\sin(\pi(1-\epsilon))}{\pi(1-\epsilon)}|x-y|,$
and hence
\begin{align*}
|\sin(2\pi x)-\sin(2\pi y)|^2
=4\cos^2\left(\pi (x+y)\right)\sin^2\left(\pi (x-y)\right)\ge 4\cos^2(\pi\epsilon)\frac{\sin^2(\pi(1-\epsilon))}{\pi^2(1-\epsilon)^2}|x-y|^2.
\end{align*}
Combining 
\textit{Case 1} and \textit{Case 2}, the proof of \eqref{sin} is completed. 
We now turn to the proof of \eqref{sinN}. 
Recall that $e_1^N$ defined in \eqref{sinN0} is the linear interpolation of points $\{e_1(j/N),\,j\in\mathbb{Z}_{N+1}\cup\{0\}\}$.  Without loss of generality, assume that $x<y$. %
We split the interval $[\epsilon,1-\epsilon]$ into 
$$\left[\epsilon,\frac{1}{2}-\frac{1}{N}\right]\cup\left(\frac{1}{2}-\frac{1}{N},\frac{1}{2}+\frac{1}{N}\right)\cup\left[\frac{1}{2}+\frac{1}{N},1-\epsilon\right]=:A_1\cup A_2\cup A_3.$$
For $N\ge 8$, it holds that $A_2\subset [\frac{1}{4}+\frac{1}{N},\frac{3}{4}-\frac{1}{N}]$.
If $x,y\in[x_{l-1},x_{l}]$ for some $l\in\mathbb{Z}_{N+1}$, then 
$$|e_i^N(x)-e_i^N(y)|=N|x-y|\left|e_i\left(\frac{l-1}{N}\right)-e_i\left(\frac{l}{N}\right)\right|,~i=1,2.$$
It follows from \eqref{sin} that
$$|e_1^N(x)-e_1^N(y)|+|e_2^N(x)-e_2^N(y)|\ge c(\epsilon)|x-y|.$$
Hence, we only need to prove the case of $x\in[x_l,x_{l+1}]$ and $y\in(x_{m},x_{m+1}]$ for some $l<m$.

(a) $x,y\in A_1$. 
By the mean value theorem, we have $$|e_1(x)-e_1(y)|\ge \sqrt{2}\cos\left(\frac{\pi}{2}-\frac{\pi}{N}\right)|x-y|\ge\sqrt{2}\pi\sin\left(\frac{\pi}{N}\right)|x-y|,$$
which together with the fact that $e_1(x)=\sqrt{2}\sin(\pi x)$ is strictly increasing  in $A_1$ yields that $$|e_1^N(x)-e_1^N(y)|\ge\sqrt{2}\pi\sin\left(\frac{\pi}{N}\right)|x-y|.$$
If $x\in[x_l,x_{l+1}]\subset A_1$ and $y\in(x_{m},x_{m+1}]\subset A_1$  for some $l<m$, then
\begin{align*}
|e_1^N(x)-e_1^N(y)|&= |e_1^N(x)-e_1^N(x_{l+1})|+|e_1^N(x_{l+1})-e_1^N(x_m)|+|e_1^N(x_{m})-e_1^N(y)|\\
&\ge \sqrt{2}\pi\sin\left(\frac{\pi}{N}\right)(x_{l+1}-x)+\sqrt{2}\pi\sin\left(\frac{\pi}{N}\right)(x_m-x_{l+1})+\sqrt{2}\pi\sin\left(\frac{\pi}{N}\right)(y-x_m)\\
&=\sqrt{2}\pi\sin\left(\frac{\pi}{N}\right)|x-y|.
\end{align*}

(b) $x,y\in [\frac{1}{4}+\frac{1}{N},\frac{3}{4}-\frac{1}{N}]$ or $x,y\in A_3$. Notice that $e_1(x)=\sqrt{2}\sin(\pi x)$ is strictly decreasing  in $A_3$ with derivative $e_1^\prime(x)=\sqrt{2}\pi\cos(\pi x)\in(-\sqrt{2}\pi\cos(\pi\epsilon),-\sqrt{2}\pi\sin\left(\frac{\pi}{N}\right))$,
and $e_2(x)=\sqrt{2}\sin(2\pi x)$  is strictly decreasing  in $[\frac{1}{4}+\frac{1}{N},\frac{3}{4}-\frac{1}{N}]$ with derivative $e_2^\prime(x)=2\sqrt{2}\pi\cos(2\pi x)\in(-2\sqrt{2}\pi,-2\sqrt{2}\pi\sin\left(\frac{2\pi}{N}\right))$. Hence, \eqref{sinN} follows from an  argument similar to (a).

(c) $x\in A_1$, $y\in A_2$. If $x\in[\frac{1}{4}+\frac{1}{N},\frac{3}{4}-\frac{1}{N}]\cap A_1$, then  $x,y\in[\frac{1}{4}+\frac{1}{N},\frac{3}{4}-\frac{1}{N}]$, and 
\eqref{sinN} holds  by virtue of (b). If $x\in[\epsilon,\frac{1}{4}+\frac{1}{N}]$, then 
\begin{align*}
|e_1^N(x)-e_1^N(y)|&\ge\sin\left(\pi\left(\frac{1}{2}-\frac{2}{N}\right)\right)-\sin\left(\frac{\pi}{4}\right)\\
&=2\cos\left(\left(\frac{3}{8}-\frac{1}{N}\right)\pi\right)\sin\left(\left(\frac{1}{8}-\frac{1}{N}\right)\pi\right)\\
&\ge 2\cos\left(\frac{\pi}{4}\right)\sin\left(\left(\frac{1}{8}-\frac{1}{N}\right)\pi\right)|x-y|,
\end{align*}
since $|x-y|\le 1$ and $N>8$.

(d) $x\in A_2$, $y\in A_3$. In this case, the proof is similar to (c).

(e) $x\in A_1$, $y\in A_3$. For $x\in A_1$, $e_2^N(x)\ge \min\left\{\sin(2\pi\epsilon),\sin(\frac{2\pi}{N})\right\}=:c_0>0$, and $e_2^N(y)\le -c_0$. Hence,
\begin{align*}
|e_2^N(x)-e_2^N(y)|\ge2c_0\ge2c_0|x-y|.
\end{align*}

The proof is finished.
\end{proof}
Recall that  the numerical solution $v^N$ of FDM or SGM for  \eqref{heatv0} is formulated by  \eqref{Heatw1}.  
Based on Lemma \ref{sincos}, we proceed to obtain the optimal H\"older continuity exponent of $(t,x)\mapsto v^N(t,x)$. 

\begin{prop}\label{Holder-h0N}
Let $N\gg1$.
 Then there exist positive constants $C_i=C_i(N,\epsilon,T_0,T),$ $i=1,2$ such that 
 for any $(t,x),\,(s,y)\in[T_0,T]\times[\epsilon,1-\epsilon]$,
\begin{align}\label{upph}
C_1(|t-s|+|x-y|^2)\le\E|v^N(t,x)-v^N(s,y)|^2\le C_2(|t-s|+|x-y|^2).
\end{align}

\end{prop}
\begin{proof}
The proof is separated into three steps.

\textit{Step 1:}
In view of \eqref{NGDG} and \eqref{NGDF}, one can check that the discrete heat kernel associated with SGM or FDM satisfies the following two facts:

$(i)$ the sequence $\{\lambda_k^N\}_{k\in\mathbb{Z}_N}\subset(-\infty,0]$ is strictly decreasing with respect to $k$;

$(ii)$ for every $N\ge 1$ and $k\in\mathbb{Z}_N$, functions
$ \varphi_k^N,\,\psi_k^N: [0,1]\rightarrow \R$ are uniformly bounded from below and above by $-\sqrt 2$ and $\sqrt 2$, respectively. Moreover, 
$$\left|\varphi_k^N(x)-\varphi_k^N(y)\right|\le\sqrt{2} \pi k|x-y|,~\forall\,x,y\in[0,1].$$
Then the proof of the right side of \eqref{upph} is standard by using the above facts $(i)$ and $(ii)$.

\textit{Step 2:} In this step, we prove the left side of \eqref{upph} for $t=s$ or $x=y$. Without loss of generality, assume that $t>s$.
Notice that for $m\neq n$, 
\begin{equation*}\label{orth}
\int_0^{1}\psi_m^N(y)\psi_n^N(y)\ud y=0,
\end{equation*}
which leads to
\begin{align}\label{N1sp}\notag
\E|v^N(t,x)-v^N(t,y)|^2=&\int_0^t\int_0^{1}|G^N_{t-r}(x,z)-G^N_{t-r}(y,z)|^2\ud r\ud z\\
=&\sum_{k=1}^{N-1}\int_0^te^{2\lambda_k^N(t-r)}\ud r|\varphi_k^N(x)-\varphi_k^N(y)|^2
\end{align}
and
\begin{align}\label{N1ti}\notag
\E|v^N(t,x)-v^N(s,x)|^2=&\int_0^s\int_0^{1}|G^N_{t-r}(x,z)-G^N_{s-r}(x,z)|^2\ud r\ud z+\int_s^t\int_0^{1}|G^N_{t-r}(x,z)|^2\ud r\ud z\\
\ge&\sum_{k=1}^{N-1}\int_s^te^{2\lambda_k^N(t-r)}\ud r|\varphi_k^N(x)|^2.
\end{align}
Noticing that for any
 $x\in[\epsilon,1-\epsilon]$, we have $\sin(\pi x)\ge\min\{\sin(\pi \epsilon),\sin(\pi (1-\epsilon))\}=\sin(\pi \epsilon)>0$, hence $\varphi^N_1(x)\ge\sin(\pi \epsilon)$. This yields that
  \eqref{N1ti} is bounded from below as
$$ \E|v^N(t,x)-v^N(s,x)|^2\ge\int_s^te^{2\lambda^N_1(t-r)}\ud r |\varphi^N_1(x)|^2\ge C(\epsilon,T)(t-s),$$
which proves the lower bound in \eqref{upph} for the case $t\neq s$ and $x=y$.

To prove the case $t=s\ge T_0$ and $x\neq y$, it is sufficient to notice that 
$$|\varphi_1^N(x)-\varphi_1^N(y)|^2+|\varphi_2^N(x)-\varphi_2^N(y)|^2\ge c(\epsilon,N)|x-y|^2,$$
by virtue of Lemma \ref{sincos}.
Substituting it into \eqref{N1sp} gives 
\begin{align*}
\E|v^N(t,x)-v^N(t,y)|^2&\ge \sum_{k=1}^2\int_0^te^{2\lambda_k^N(t-r)}\ud r|\varphi_k^N(x)-\varphi_k^N(y)|^2\\
&\ge \frac{e^{2\lambda_2^N t}-1}{2\lambda_2^N}\sum_{k=1}^2|\varphi_k^N(x)-\varphi_k^N(y)|^2\ge c(\epsilon,T_0,N)|x-y|^2.
\end{align*}

\textit{Step 3:}
In \textit{Step 2}, we have shown that there exist $c_i,\,i=1,2,3,4$ such that for any $\,t\in[T_0,T]$,
$$c_1|x-y|^2\le \E|v^N(t,x)-v^N(t,y)|^2\le c_2|x-y|^2,~\forall\,x,y\in[\epsilon,1-\epsilon],$$
and that for any $x\in[\epsilon,1-\epsilon]$,
 $$c_3|t-s|\le \E|v^N(t,x)-v^N(s,x)|^2\le c_4|t-s|,~\forall\,t,s\in[T_0,T].$$
 In order to extend the lower bound in \eqref{upph} to the case of $t\neq s$ and $x\neq y$, it suffices to discuss the following three situations:
 
  $(1)~ |x-y|^2\ge\frac{4c_4}{c_1}|t-s|$;
  
$(2)~|t-s|\ge\frac{4c_2}{c_3}|x-y|^2$;

$(3)~\frac{c_1}{4c_4}|x-y|^2\le |t-s|\le \frac{4c_2}{c_3}|x-y|^2$.\\
Readers are referred to \cite[Lemma 5.1]{HLS16} for detailed discussions for the above three situations. The proof is finished.

\end{proof}

\textit{Proof of Theorem \ref{finite-space-hit}:} 
By Proposition \ref{Holder-h0N}, the optimal H\"older continuity exponent of the considered numerical solution $v^N$  is  $\frac{1}{2}$ in time direction.
For any fixed $x\in[\epsilon,1-\epsilon]$,
 $$\Var \,v^N(t,x)= \sum_{k=1}^{N-1}\int_0^te^{2\lambda_{k}^N(t-r)}\ud r|\varphi^N_k(x)|^2,~\forall\,t\in[T_0,T].$$
Thus, $\Var\, v^N(\cdot,x)$ is a Lipschitz function with respect to $t\in[T_0,T]$.  
Besides, we notice that
\begin{align}\notag\label{C0} 
\Var \,v^N(t,x)&\ge
 \sum_{k=1}^{N-1}\int_0^te^{2\lambda_{k}^N(t-r)}\ud r|\varphi^N_k(x)|^2\\
 &\ge \frac{e^{2\lambda_{1}^NT_0}-1}{\lambda_{1}^N}\sin^2(\pi\epsilon),~\forall\,(t,x)\in[T_0,T]\times[\epsilon,1-\epsilon].
\end{align}
By Remark \ref{rem1} and Proposition \ref{Holder-h0N},  for any $x\in[\epsilon,1-\epsilon]$, \begin{equation}\label{C2spa}
\Var\left(v^N(t,x)|v^N(s,x)\right)\ge c\sqrt{t-s},~\forall\,t,s\in[T_0,T].
\end{equation} 
Based on Proposition \ref{Holder-h0N}, \eqref{C0} and \eqref{C2spa}, we finish the proof of Theorem \ref{finite-space-hit}. 
\qed

\subsection{Proof of Theorem \ref{spectral-space-hit}}

 In this part, let $v^N(t,x)$ be the numerical solution of SGM for \eqref{heatv0}. 
 By \cite[Proposition 3.13]{EML07}, we have
$$\Var \left(v^N(t,x)|v^N(s,y)\right)=\frac{ \Var\, v^N(t,x)\Var\, v^N(s,y)-\Cov( v^N(t,x), v^N(s,y))^2}{ \Var\, v^N(s,y)}.$$
 Based on \eqref{C0}, we proceed to 
 derive the lower bound of $\Var \left(v^N(t,x)|v^N(s,y)\right)$. 

\begin{prop}\label{spectralC2}
Let $v^N(t,x)$ be the numerical solution of SGM for \eqref{heatv0}. Then
for any $0<\epsilon\ll 1$,
there exists $N_0:=N_0(\epsilon)$ such that for all $N\ge N_0$, 
\begin{align*}
 \Var\, v^N(t,x)\Var\, v^N(t,y)-\Cov( v^N(t,x), v^N(t,y))^2\ge c(\epsilon,N,T_0)|x-y|^2
\end{align*}
holds for any $x,y\in[\epsilon,1-\epsilon]$ and $t\in[T_0,T]$.

\end{prop}
\begin{proof}
Without loss of generality, assume that $\epsilon\le y<x\le 1-\epsilon$ and $t\in[T_0,T]$ be arbitrarily fixed. First,
we claim that for any $N\ge 2$, 
\begin{align}\label{Var2-Cov2}
\Var\, v^N(t,x)\Var \,v^N(t,y)-\Cov( v^N(t,x), v^N(t,y))^2>0,~\forall\,x\neq y.
\end{align}
Otherwise, there exist $x\neq y$ and $\lambda_0\in\R$ such that
$ v^N(t,x)=\lambda_0 v^N(t,y)$ a.s. Hence,
\begin{align}\label{positive}\notag
&\int_0^{t}\int_0^1|G^N_{t-r}(x,z)-\lambda_0G^N_{t-r}(y,z)|^2\ud z\ud r\\
=&\sum_{k=1}^{N-1}\int_0^{t}e^{-2k^2\pi^2(t-r)}\ud r| e_k(x)-\lambda_0e_k(y)|^2=0,
\end{align}
which implies that
$$0\le \frac{\sin(\pi x)}{\sin(\pi y)}=\frac{\sin(2\pi x)}{\sin(2\pi y)}=\frac{\sin(3\pi x)}{\sin(3\pi y)}=\cdots=\frac{\sin((N-1)\pi x)}{\sin((N-1)\pi y)}.$$
By the elementary identities $\sin(2\pi x)=2\sin(\pi x)\cos(\pi x)$ and $\sin(3\pi x)=3\sin(\pi x)-4\sin^3(\pi x)$, it must hold that $\cos(\pi x)=\cos(\pi y)$ and $\sin^2(\pi x)=\sin^2(\pi y)$. However, this only occurs when $x=y$ since $x,y\in(0,1)$. Thus, \eqref{Var2-Cov2} is valid.

By denoting $$a_k=\left(\int_0^te^{-2k^2\pi^2(t-r)}\ud r\right)^{\frac{1}{2}}e_k(x),~b_k=\left(\int_0^te^{-2k^2\pi^2(t-r)}\ud r\right)^{\frac{1}{2}}e_k(y),$$
 we rewrite
\begin{align}\label{Var2-Cov}\notag
&\Var \,v^N(t,x)\Var \,v^N(t,y)-\Cov( v^N(t,x), v^N(t,y))^2\\
=&\left(\sum_{i=0}^N a_i^2\right)\left(\sum_{j=0}^N b_j^2\right)-\left|\sum_{i=0}^Na_ib_i\right|^2
=\sum_{i<j}|a_ib_j-a_jb_i|^2.
\end{align}
By definitions of $a_k$ and $b_k$, we have 
$$|a_ib_j-a_jb_i|^2=\frac{1}{4}K_{i,j}^N(t)
|e_i(x)e_j(y)-e_j(x)e_i(y)|^2,$$
where $$K_{i,j}^N(t)
:=\frac{(e^{-2\pi^2 i^2t}-1)(e^{-2\pi^2 j^2t}-1)}{\pi^4 i^2 j^2}.$$
Obviously, $0<K_{i,j}^N(T_0)\le K_{i,j}^N(t)\le K_{i,j}^N(T)<\infty,~\forall\,t\in[T_0,T]$.
Notice that
\begin{align}\label{sincos00}\notag
&\sin(i\pi x)\sin(j\pi y)-\sin(i\pi y)\sin(j\pi x)\\
=&\sin\frac{(i+j)\pi(x-y)}{2}\sin\frac{(j-i)\pi(x+y)}{2}-\sin\frac{(i+j)\pi(x+y)}{2}\sin\frac{(j-i)\pi(x-y)}{2}.
\end{align}
Choose $N_0:=N_0(\epsilon)>2$ such that $(2N_0-3)\sin(\pi \epsilon)\ge \frac{\pi}{2}+1$ and let $N\ge N_0$ be arbitrarily fixed. 
For any $y\in[\epsilon,1-\epsilon]$, denote $y_{N}^\epsilon:=y+\frac{1}{2N-3}$. 
By  \eqref{Var2-Cov}, we have
\begin{align*}
&\quad\Var \,v^N(t,x)\Var \,v^N(t,y)-\Cov( v^N(t,x), v^N(t,y))^2\\
&\ge |a_{N-2}b_{N-1}-a_{N-1}b_{N-2}|^2\\
& \ge K_{N-2,N-1}^N(T_0)\left|\sin\frac{(2N-3)\pi(x-y)}{2}\sin\frac{\pi(x+y)}{2}-\sin\frac{(2N-3)\pi(x+y)}{2}\sin\frac{\pi(x-y)}{2}\right|^2.
\end{align*}
We are going to derive the lower bound of $\Var \,v^N(t,x)\Var \,v^N(t,y)-\Cov( v^N(t,x), v^N(t,y))^2$, which is separated  into two cases.  

\textit{Case 1:} $y\in [\epsilon,1-\epsilon)$ and $x\in(y,y_{N}^\epsilon]\cap(\epsilon,1-\epsilon]$. We
 introduce
\begin{align*}
f_N(x,y):&=\sin\frac{(2N-3)\pi(x-y)}{2}\sin\frac{\pi(x+y)}{2}-\sin\frac{(2N-3)\pi(x+y)}{2}\sin\frac{\pi(x-y)}{2}\\
&\ge \sin\frac{(2N-3)\pi(x-y)}{2}\sin(\pi \epsilon)-\sin\frac{\pi(x-y)}{2}\\
&\ge(2N-3)(x-y)\sin(\pi \epsilon)-\frac{\pi}{2}(x-y),
\end{align*} 
where we have used $0<x-y\le\frac{1}{2N-3}$, and $\frac{2}{\pi}z\le\sin(z)\le z,~\forall\,z\in(0,\frac{\pi}{2}]$. Hence, for any $N\ge N_0$ with $(2N_0-3)\sin(\pi \epsilon)-\frac{\pi}{2}\ge 1$,
\begin{align*}
f_N(x,y)\ge(x-y),
\end{align*} 
which implies that 
\begin{align}\label{spa1}
&\Var \,v^N(t,x)\Var \,v^N(t,y)-\Cov( v^N(t,x), v^N(t,y))^2
\ge c(T_0,N)|x-y|^2.
\end{align}

\textit{Case 2:} $y\in [\epsilon,1-\epsilon)$ and $x\in(y_{N}^\epsilon,1-\epsilon]$. By \eqref{Var2-Cov2} and the continuity of $\Var \,v^N(t,x)$ and $\Cov( v^N(t,x), v^N(t,y))$, we have that there is $c=c(\epsilon,N)$  such that
\begin{align}\label{spa2}
\Var \,v^N(t,x)\Var \,v^N(t,y)-\Cov( v^N(t,x), v^N(t,y))^2\ge c\ge \frac{c}{(1-2\epsilon)^2}|x-y|^2.
\end{align}
Combining \eqref{spa1} and \eqref{spa2}, we finish the proof.

\end{proof}

\textit{Proof of Theorem \ref{spectral-space-hit}:}
The proof of \eqref{exas3} is similar to that of Theorem 
\ref{finite-space-hit}. The proof of \eqref{exas1} follows from Proposition \ref{Holder-h0N}, \eqref{C0} and Proposition \ref{spectralC2}.
\qed

\subsection{Proof of Theorem \ref{EE-hit}}
Recall that the numerical solution $v_M$  of EEM for \eqref{heatv0}  is defined in \eqref{HeatM}. We begin with giving the optimal H\"older continuity of $v_M(t_i,\cdot).$
\begin{lem}\label{Holder-h1Nx}
Let $0< \epsilon\ll1$ be fixed and $M\ge3$. Then there exist  positive constants $C_j=C_j(\epsilon,M,T),\,j=1,2$ such that for any $1\le i\le M$,
\begin{align*}
C_1|x-y|^2\le \E|v_M(t_i,x)-v_M(t_i,y)|^2\le C_2|x-y|^2,~\forall\,x,y\in[\epsilon,1-\epsilon].
\end{align*}
\end{lem}

\begin{proof}
First, it follows from \eqref{GHP} that
\begin{align*}
\E|v_M(t_i,x)-v_M(t_i,y)|^2=&\int_{0}^{t_i}\int_0^1|G_{t_i-[\frac{r}{\delta t}]\delta t}(x,z)-G_{t_i-[\frac{r}{\delta t}]\delta t}(y,z)|^2\ud z\ud r\\
\le&\sum_{k=0}^{i-1}\int_{t_k}^{t_{k+1}}\int_0^1\left|P_{t_i-t_k}(x,z)-P_{t_i-t_k}(y,z)\right|^2\ud z\ud r+C|x-y|^2,
\end{align*}
since $H$ is smooth.
Because $xe^{-x}$ is uniformly bounded on $[0,\infty)$,  there is $C>0$ such that for any $\xi\in[0,1]$ and $s> 0$,
\begin{align*}
\int_0^1\left|\frac{\ud }{\ud \xi}P_s(\xi,z)\right|^2\ud z=\int_0^1\frac{1 }{8\pi s^2}\exp\left(-\frac{(\xi-z)^2}{2s}\right)\frac{(\xi-z)^2}{2s}\ud z\le \frac{C}{s^2}.
\end{align*}
Therefore,
applying the mean value theorem yields
\begin{align*}
&\quad\sum_{k=0}^{i-1}\int_{t_k}^{t_{k+1}}\int_0^1\left|P_{t_i-t_k}(x,z)-P_{t_i-t_k}(y,z)\right|^2\ud z\ud r\\
&\le C|x-y|^2\sum_{k=0}^{i-1}\int_{t_k}^{t_{k+1}}\frac{1}{(t_i-t_k)^2}\ud r\\
&= C|x-y|^2\left(\frac{1}{\delta t}+\sum_{k=0}^{i-2}\int_{t_k}^{t_{k+1}}\frac{1}{(t_i-t_k)^2}\ud r\right)\\
&\le C|x-y|^2\left(\frac{1}{\delta t}+\int_{0}^{t_{i-1}}\frac{1}{(t_i-r)^2}\ud r\right)\le \frac{2C}{\delta t}|x-y|^2.
\end{align*}
On the other hand, by the spectral expansion of $G$ and Lemma \ref{sincos},
\begin{align*}
\E|v_M(t_i,x)-v_M(t_i,y)|^2&=\int_{0}^{t_i}\int_0^1\left|G_{t_i-[\frac{r}{\delta t}]\delta t}(x,z)-G_{t_i-[\frac{r}{\delta t}]\delta t}(y,z)\right|^2\ud z\ud r\\
&=\sum_{k=1}^\infty \int_{0}^{t_i}e^{-2k^2\pi^2(t_i-[\frac{r}{\delta t}]\delta t)}\ud r|e_k(x)-e_k(y)|^2\\
&\ge c(\epsilon,T_0,T)|x-y|^2.
\end{align*}
The proof is completed.
\end{proof}

\textit{Proof of Theorem \ref{EE-hit}:}
For $(t,x)\in[T_0,T]\times[\epsilon,1-\epsilon]$,
\begin{equation*}
\Var \,v_M(t,x)= \sum_{k=1}^\infty\int_0^te^{-2\pi^2k^2(t-[\frac{r}{\delta t}]\delta t)}\ud r|e_k(x)|^2\ge c\sin^2(\pi\epsilon)
\end{equation*}
with $c=\int_0^{T_0}e^{-2\pi^2T}\ud r>0$. 

By introducing $$a_k=\left(\int_0^te^{-2k^2\pi^2({t-[\frac{r}{\delta t}]\delta t})}\ud r\right)^{\frac{1}{2}}e_k(x),~b_k=\left(\int_0^te^{-2k^2\pi^2({t-[\frac{r}{\delta t}]\delta t})}\ud r\right)^{\frac{1}{2}}e_k(y), ~ k\in\mathbb{N}$$
and
$$K_{i,j}^N(t)
:=\left(\int_0^te^{-2i^2\pi^2({t-[\frac{r}{\delta t}]\delta t})}\ud r\right)\left(\int_0^te^{-2j^2\pi^2({t-[\frac{r}{\delta t}]\delta t})}\ud r\right),~ i,j\in\mathbb{N},$$
and repeating the proof of Proposition \ref{spectralC2}, one can verify that
$$\Var \,v_M(t_i,x)\Var \,v_M(t_i,y)-\Cov(v_M(t_i,x),v_M(t_i,y))\ge c(T_0,\epsilon)|x-y|^2,~\forall\,x,y\in[\epsilon,1-\epsilon].$$ 
Finally, the proof of Theorem \ref{EE-hit} is completed by Lemma \ref{Holder-h1Nx}.
\qed

\subsection{Proofs of results in Section \ref{SEC3.2}}
In this part, we give proofs of Lemma \ref{Holder-h1Nt},  Corollary \ref{cor1} and Proposition \ref{OUn}.

\textit{Proof of Lemma \ref{Holder-h1Nt}:}
In view of \eqref{HeatM}, we have
\begin{align}\label{holdM}\notag
&\E|v_M(t_i,x)-v_M(t_j,x)|^2\\=&\int_{0}^{t_j}\int_0^1|G_{t_i-[\frac{r}{\delta t}]\delta t}(x,z)-G_{t_j-[\frac{r}{\delta t}]\delta t}(x,z)|^2\ud z\ud r+\int_{t_j}^{t_{i}}\int_0^1|G_{t_i-[\frac{r}{\delta t}]\delta t}(x,z)|^2\ud z\ud r.
\end{align}
Thus,
to prove the lower bound of \eqref{upphM}, it suffices to prove that there is $c>0$ such that
\begin{equation}\label{kik}
\sum_{k=j}^{i-1}\int_{t_k}^{t_{k+1}}\int_0^1\left|G_{t_i-t_k}(x,z)\right|^2\ud z\ud r\ge c\sqrt{t_i-t_j}.
\end{equation}
In fact, by the elementary property $\int_0^1\left|G_{r}(x,z)\right|^2\ud z=G_{2r}(x,x)$ and \eqref{PGP},
\begin{align*}
&\quad\sum_{k=j}^{i-1}\int_{t_k}^{t_{k+1}}\int_0^1\left|G_{t_i-t_k}(x,z)\right|^2\ud z\ud r=\sum_{k=j}^{i-1}\int_{t_k}^{t_{k+1}}G_{2(t_i-t_k)}(x,x)\ud r \\
&\ge C\sum_{k=j}^{i-1}\int_{t_k}^{t_{k+1}}P_{2(t_i-t_k)}(x,x)\ud r \ge C\sum_{k=1}^{i-j}\frac{\delta t}{\sqrt{k\delta t}}\ge \int_{t_1}^{t_{i-j+1}}\frac{C}{\sqrt{r}}\ud r\\&=C\left(\sqrt{t_{i-j+1}}-\sqrt{t_{1}}\right).
\end{align*}
For $j=i-1$, $\sqrt{t_{i-j+1}}-\sqrt{t_{1}}=(\sqrt{2}-1)\sqrt{t_i-t_j}$, and for $1\le j<i-1$,
$\sqrt{t_{i-j+1}}-\sqrt{t_{1}}\ge \frac{1}{2}\sqrt{t_i-t_j}$. Thus, we obtain \eqref{kik}, which yields the left side of \eqref{upphM}.
Similarly, we also have 
\begin{align*}
\sum_{k=j}^{i-1}\int_{t_k}^{t_{k+1}}\int_0^1\left|G_{t_i-t_k}(x,z)\right|^2\ud z\ud r
&\le C\sum_{k=j}^{i-1}\int_{t_k}^{t_{k+1}}P_{2(t_i-t_k)}(x,x)\ud r \\
&\le C \int_{T_0}^{t_{i-j}}\frac{1}{\sqrt{r}}\ud r=2C\sqrt{t_i-t_j}.
\end{align*}
For the right side of \eqref{upphM}, it remains to estimate the first term on the right side of \eqref{holdM}. Notice that  by \eqref{GD0}, $G_t(x,x)$ is decreasing with respect to the variable $t$. Therefore,
\begin{align*}
&\quad\int_0^1|G_{t_i-[\frac{r}{\delta t}]\delta t}(x,z)-G_{t_j-[\frac{r}{\delta t}]\delta t}(x,z)|^2\ud z\\
 &=G_{2(t_i-[\frac{r}{\delta t}]\delta t)}(x,x)-G_{t_i+t_j-2[\frac{r}{\delta t}]\delta t}(x,x)+G_{2(t_j-[\frac{r}{\delta t}]\delta t)}(x,x)-
G_{t_i+t_j-2[\frac{r}{\delta t}]\delta t}(x,x)\\
&\le G_{2(t_j-[\frac{r}{\delta t}]\delta t)}(x,x)-
G_{t_i+t_j-2[\frac{r}{\delta t}]\delta t}(x,x),
\end{align*}
which in combination with \eqref{GHP} yields
\begin{align}\label{inte}\notag
&\quad\int_0^{t_j}\int_0^1|G_{t_i-[\frac{r}{\delta t}]\delta t}(x,z)-G_{t_j-[\frac{r}{\delta t}]\delta t}(x,z)|^2\ud z\ud r\\\notag
&\le\int_0^{t_j}P_{2(t_j-[\frac{r}{\delta t}]\delta t)}(x,x)-P_{t_i+t_j-2[\frac{r}{\delta t}]\delta t}(x,x)\ud r+C(t_i-t_j)\\
&\le C\sum_{k=0}^{j-1}\delta t\Big(\frac{1}{\sqrt{t_j-t_k}}-\frac{1}{\sqrt{\frac{t_j+t_i}{2}-t_k}}\Big)+C(t_i-t_j).
\end{align}
The first term on the right side of \eqref{inte} is bounded as
\begin{align*}
&\quad\sum_{k=0}^{j-1}\delta t\Big(\frac{1}{\sqrt{t_j-t_k}}-\frac{1}{\sqrt{\frac{t_j+t_i}{2}-t_k}}\Big)\le
\sum_{k=0}^{j-1}\delta t\Big(\frac{1}{\sqrt{t_j-t_k}}-\frac{1}{\sqrt{t_i-t_k}}\Big)\\
&=\sum_{k=0}^{j-1}\int_{t_k}^{t_{k+1}}\int_{t_j}^{t_i}-\frac{\ud}{\ud s}\frac{1}{\sqrt{s-t_{k}}}\ud s \ud u
\le\sum_{k=0}^{j-1}\int_{t_k}^{t_{k+1}}\int_{t_j}^{t_i}-\frac{\ud}{\ud s}\frac{1}{\sqrt{s-u}}\ud s \ud u\\
&=\int_{t_j}^{t_i}\int_{t_0}^{t_{j}}\frac{\ud}{\ud u}\frac{1}{\sqrt{s-u}}\ud u\ud s 
=2\sqrt{t_i-t_{j}}-2(\sqrt{t_i}-\sqrt{t_j})\\
&\le 2\sqrt{t_i-t_j}.
\end{align*}
The proof is completed.
\qed

\textit{Proof of Corollary \ref{cor1}:} 

First, notice that there is $j\le i+1$ such that $t\in[t_i,t_{i+1})$ and $s\in[t_{j-1},t_{j})$.
\begin{align*}
&\quad \E|v_M(t,x)-v_M(s,x)|^2\\
&=\int_{0}^{s}\int_0^1|G_{t-[\frac{r}{\delta t}]\delta t}(x,z)-G_{s-[\frac{r}{\delta t}]\delta t}(x,z)|^2\ud z\ud r+\int_{s}^{t}\int_0^1\left|G_{t-t_i}(x,z)\right|^2\ud z\ud r\\
&=:A(s,t)+B(s,t).
\end{align*}

If $i=j-1$,  then $t_i\le s< t\le t_{i+1}$.
Similar to the first term on the right side of \eqref{holdM}, we also have
$A(s,t)\le C\sqrt{t-s},$
which together with  the fact $B(s,t)\le C\frac{t-s}{\sqrt{t-t_i}}\le C\sqrt{t-s}$
 completes the proof of \eqref{vmst0}.

If $i\ge j$, then by Lemma \ref{Holder-h1Nt},
\begin{align*} 
&\quad\E|v_M(t,x)-v_M(s,x)|^2\\&\le3\E|v_M(t,x)-v_M(t_i,x)|^2+3\E|v_M(t_i,x)-v_M(t_j,x)|^2+3\E|v_M(t_j,x)-v_M(s,x)|^2\\
&\le C\sqrt{t-t_i} +c_1\sqrt{t_i-t_j}+C\sqrt{t_j-s}\le (2C+c_1)\sqrt{t-s}.
\end{align*}

Assume by contradiction  that there is $c_4>0$ such that \eqref{vmst1} holds. 
Fix $t\in[t_i+\frac{\delta t}{2},t_{i+1})$ and let $s_n=t-\frac{1}{n},\,n\ge 1$. 
Similar to \eqref{inte}, it holds that
\begin{align*}
A(s_n,t)&=\int_{0}^{s_n}\int_0^1|G_{t-[\frac{r}{\delta t}]\delta t}(x,z)-G_{s_n-[\frac{r}{\delta t}]\delta t}(x,z)|^2\ud z\ud r\\
&\le\int_0^{s_n}P_{2(s_n-[\frac{r}{\delta t}]\delta t)}(x,x)-P_{t+s_n-2[\frac{r}{\delta t}]\delta t}(x,x)\ud r+C(t-s_n)\\
&\le \sum_{k=0}^{i}C\delta t\Big(\frac{1}{\sqrt{s_n-t_k}}-\frac{1}{\sqrt{\frac{t+s_n}{2}-t_k}}\Big)+C(t-s_n).
\end{align*}
A direct calculation by using L'Hospital's rule gives that for any $k=0,1,\ldots,i,$
$$\lim_{n\rightarrow \infty}\frac{\frac{1}{\sqrt{s_n-t_k}}-\frac{1}{\sqrt{\frac{t+s_n}{2}-t_k}}}{\sqrt{t-s_n}}=\lim_{s\rightarrow t^-}\frac{\frac{1}{\sqrt{s-t_k}}-\frac{1}{\sqrt{\frac{t+s}{2}-t_k}}}{\sqrt{t-s}}=0,$$
which indicates that
$\lim_{n\rightarrow\infty}\frac{A(s_n,t)}{\sqrt{t-s_n}}=0.$
On the other hand, we also have
$$\frac{B(s_n,t)}{\sqrt{t-s_n}}=\frac{\int_{s_n}^{t}\int_0^1\left|G_{t-t_i}(x,z)\right|^2\ud z\ud r}{\sqrt{t-s_n}}\le C\frac{\sqrt{t-s_n}}{\sqrt{t-t_i}}\rightarrow 0,~\text{as}~n\rightarrow \infty. $$
In conclusion, we have shown 
$\lim\limits_{n\rightarrow\infty}\frac{\E|v_M(t,x)-v_M(s_n,x)|^2}{\sqrt{t-s_n}}=0,$
which contradicts \eqref{vmst1}.
The proof is finished.
\qed

\textit{Proof of Proposition \ref{OUn}:}

By \eqref{YOt},
It\^o's isometry gives
  \begin{align*}
  \E|\widetilde Y^0(t)-\widetilde Y^0(s)|^2
  =&\int_s^t(1-\lambda\delta t)^{2[\frac{t-r}{\delta t}]}\ud r+\int_0^s|(1-\lambda\delta t)^{[\frac{t-r}{\delta t}]}-(1-\lambda\delta t)^{[\frac{s-r}{\delta t}]}|^2\ud r.
 \end{align*}
Since $1-\lambda\delta t\in(0,1)$, we have that for any $T_0\le s<t\le T$,
 $$(1-\lambda\delta t)^{2[\frac{T-T_0}{\delta t}]}(t-s)\le \int_s^t(1-\lambda\delta t)^{2[\frac{t-r}{\delta t}]}\ud r\le (t-s)$$
and
 \begin{align}\notag\label{I}
 \int_0^s|(1-\lambda\delta t)^{[\frac{t-r}{\delta t}]}-(1-\lambda\delta t)^{[\frac{s-r}{\delta t}]}|^2\ud r&\le\int_0^s(1-\lambda\delta t)^{2[\frac{s-r}{\delta t}]}-(1-\lambda\delta t)^{2[\frac{t-r}{\delta t}]}\ud r\\
 &\le\int_0^s1-(1-\lambda\delta t)^{2[\frac{t-r}{\delta t}]-2[\frac{s-r}{\delta t}]}\ud r=:I.
\end{align}
The estimation of $I$ is divided into two cases.

\textit{Case 1:} $t-s\ge \delta t$. It holds that $[\frac{t-r}{\delta t}]-[\frac{s-r}{\delta t}]\le \frac{t-r}{\delta t}-\frac{s-r}{\delta t}+1\le 2\frac{t-s}{\delta t},~\forall\, r\in(0,s).$ This implies that
 \begin{align*}
 I&\le\int_0^s1-(1-\lambda\delta t)^{\frac{4(t-s)}{\delta t}}\ud r=s\int_0^{\frac{4(t-s)}{\delta t}}-\frac{\ud}{\ud a}(1-\lambda\delta t)^a\ud a\\&=s\int_0^{\frac{4(t-s)}{\delta t}}(1-\lambda\delta t)^a\log\frac{1}{1-\lambda\delta t}\ud a\le s\log\frac{1}{1-\lambda\delta t}\frac{4(t-s)}{\delta t}\le C(\delta t,T,\lambda)(t-s).
 \end{align*}
 
\textit{Case 2:} $0<t-s<\delta t$. Note that $s\in[t_m,t_{m+1})$ for some $0\le m\le M-1$. Then 
 \begin{align*}
 I=&\int^s_{t-\delta t}+\int^{t-\delta t}_{s-\delta t}+\int_{t-2\delta t}^{s-\delta t}+\cdots+\int^{t-m\delta t}_{s-m\delta t}+\int^{s-m\delta t}_{\max\{0,t-(m+1)\delta t\}}\\&+\int_{0}^{\max\{0,t-(m+1)\delta t\}}1-(1-\lambda\delta t)^{2[\frac{t-r}{\delta t}]-2[\frac{s-r}{\delta t}]}\ud r.
 \end{align*}
Notice that for $i\in\{0,1,\ldots,m\}$, we have $[\frac{t-r}{\delta t}]=[\frac{s-r}{\delta t}]=i$, $\forall\,r\in(t-(i+1)\delta t,s-i\delta t)$. Therefore,
 \begin{align*}
 I&=\sum_{i=1}^m\int^{t-i\delta t}_{s-i\delta t}+\int_{0}^{\max\{0,t-(m+1)\delta t\}}1-(1-\lambda\delta t)^{2[\frac{t-r}{\delta t}]-2[\frac{s-r}{\delta t}]}\ud r\\
 &\le (m+1)(t-s)\le \frac{T}{\delta t}(t-s),
 \end{align*}
where $t-(m+1)\delta t<t-s$ is used. Combining (i) and (ii) yields $I\le C(T,\delta t,\lambda)(t-s)$.

By Remark \ref{rem1}, if we can prove that $t\mapsto \E|\widetilde Y^0(t)|^2$
is H\"older continuous with exponent $H_0>1/2$, then $\Var(\widetilde Y^0(t)|\widetilde Y^0(s))\ge C\sqrt{t-s}.$ Indeed, it follows immediately from the above estimates on \eqref{I} that 
  \begin{align*}
 \left| \E|\widetilde Y^0(t)|^2-  \E|\widetilde Y^0(s)|^2\right|
  &=\left|\int_s^t(1-\lambda\delta t)^{2[\frac{t-r}{\delta t}]}\ud r+\int_0^s(1-\lambda\delta t)^{2[\frac{s-r}{\delta t}]}-(1-\lambda\delta t)^{2[\frac{t-r}{\delta t}]}\ud r\right|\\
  &\le (t-s)+I\\
  &\le C(T,\delta t,\lambda)(t-s).
 \end{align*}
 Finally, we complete the proof by Theorem \ref{Gausshit}.
\qed

\bibliography{mybibfile}

\end{document}